\documentclass[11pt,a4paper]{amsart}
\usepackage[utf8]{inputenc}

\usepackage{amsmath}
\usepackage{amssymb}
\usepackage{enumerate}
\usepackage{setspace}
\usepackage{caption}
\usepackage{mathrsfs}
\usepackage{hyperref}
\usepackage{esint}
\usepackage{amssymb}
\usepackage{hyperref}
\usepackage{comment}
\usepackage[normalem]{ulem}
\usepackage[all]{xy}

\makeatletter
\newcommand{\dashedrightarrow}[1][2pt]{%
  \settowidth{\@tempdima}{$\rightarrow$}\rightarrow
  \makebox[-\@tempdima]{\hskip-1.5ex\color{white}\rule[0.5ex]{#1}{1pt}}
  \phantom{\rightarrow}
}
\makeatother

\usepackage{color}
\usepackage{lipsum}

\newtheorem{theorem}{Theorem}[section]
\newtheorem{lemma}[theorem]{Lemma}
\newtheorem{corollary}[theorem]{Corollary}
\newtheorem{proposition}[theorem]{Proposition}

\numberwithin{equation}{section}

\theoremstyle{definition}
\newtheorem{definition}[theorem]{Definition}
\newtheorem{remark}[theorem]{Remark}

\DeclareMathOperator{\scal}{scal}

\DeclareMathOperator{\Ric}{Ric}

\DeclareMathOperator{\sys}{sys}

\title[Balancing the 2-systoles of some psc Kähler manifolds]{Balancing the 2-systoles of some Kähler manifolds with positive scalar curvature}
\author{Hugues Auvray}
\address{Laboratoire de Mathématiques d'Orsay, UMR 8628, Université Paris-Saclay, CNRS}
\author{Thomas Richard}
\address{Laboratoire d'Analyse et de Mathématiques Appliquées, UMR 8050, Université Paris-Est Créteil, CNRS}
\date{}

\begin{document}

\maketitle
\begin{abstract}
    Theorems by Bray-Brendle-Neves and Zhu about positive scalar curvature metrics on products of a $2$-sphere and an $n$-torus suggest that positive scalar curvature and appropriate topological assumptions should lead to the existence of topologically non trivial 2-spheres of small area, which can be stated as an upper bound on the 2-systole of such manifolds.
    
    Recent progress has been made in this direction by Sha and Tsiamis under the additional Kähler assumption while Cecchini-Hirsh-Zeidler,  Stryker and Tsiamis showed similar upper bounds on the stable 2-systole using index theoretic methods.
    
    We prove here similar inequalities for some Kähler manifolds which control the relative sizes of the representatives of a well chosen set of homology classes. For instance on $(\mathbb{CP}^1\times\mathbb{CP}^1,\omega)$ with a positive scalar curvature Kähler metric we quantitatively show that largeness of one factor imposes smallness of the other.
\end{abstract}

Since the early 2010's, the study of positive scalar curvature metrics has been revived with a new focus on quantitative results (in contrast with the topological results obtained in the 80's). The usual metric invariants (volume and diameter) cannot be controlled by the scalar curvature alone as the example of $\mathbb{S}^2\times\mathbb{R}$ shows. Thus one needs to investigate more subtle invariants. For instance Bray-Brendle-Neves in dimension 3 \cite{Bray} and Zhu in dimension at most 7 \cite{Zhu} have shown the following result:

\begin{theorem}\label{thmzhu}
    Let $2\leq n\leq 7$ and let $(\mathbb{S}^2\times\mathbb{T}^{n-2},g)$ have $\scal_g\geq 2$. Then there exists a 2-sphere $\Sigma^2$ homotopic to $\mathbb{S}^2\times\{\ast\}$ whose area is at most $4\pi$. 
    
    Moreover if any $2$-sphere homotopic to $\mathbb{S}^2\times\{\ast\}$ has area at least $4\pi$ then the universal cover of $(\mathbb{S}^2\times\mathbb{T}^{n-2},g)$  is isometric to the product of the round metric on $\mathbb{S}^2$ and a flat metric on $\mathbb{R}^{n-2}$.
\end{theorem}
Note that for $n=2$ this follows immediately from the Gauss-Bonnet theorem since in this case the scalar curvature is twice the Gauss curvature.

The proof of this result uses stable minimal hypersurfaces. The theorem is actually slightly stronger as it applies to Riemannian manifolds $(M^n,g)$ with a non-zero degree map $M^n\to \mathbb{S}^2\times\mathbb{T}^{n-2}$. 

Following Berger \cite{Berger}, the systole $\sys_h(g)$ of a homology class $h\in H_k(M)$ with $(M^n,g)$ a Riemannian manifold is the infimum the areas of smooth representatives of $h$.
Since $\pi_2(\mathbb{S}^2\times\mathbb{T}^{n-2})\simeq H_2(\mathbb{S}^2)\subset H_2(\mathbb{S}^2\times\mathbb{T}^{n-2})$ by the Hurewicz theorem, the conclusion of Theorem \ref{thmzhu} can then be restated as 
\[\sys_{[\mathbb{S}^2\times\{\ast\}]}\leq 4\pi\]
with the equality case covered.

This suggests that lower bounds on the scalar curvature can under appropriate topological assumptions provide an upper bound on the systole of some classes $h\in H_2(M)$ or for $\sys_2(g)=\inf_{h\in H_2(M)\backslash\{0\}}\sys_h(g)$. Replacing $\mathbb{S}^2\times\mathbb{T}^{n-2}$ with other manifolds in Theorem \ref{thmzhu} has remained elusive for a long time, although some progress on the case of $\mathbb{S}^2\times\mathbb{S}^2$ was made in \cite{Richard}.

More recently, the case of Kähler manifolds was investigated by Z. Sha in \cite{Sha1} and \cite{Sha2} who showed an optimal upper bound for $\sys_2(\omega)$ for a closed Kähler surface $(X^2,\omega)$ with positive average scalar curvature. This relies on the cohomological nature of the 2-systoles and the average scalar curvature for Kähler manifolds.

A few months later, S. Cecchini, S. Hirsch and R. Zeidler \cite{CHZ} as well as D. Stryker \cite{Stryker} used Dirac operator methods to prove upper bounds on the slightly weaker stable 2-systole of Riemannian positive scalar curvature metrics on various manifolds, including $\mathbb{CP}^n$.

Even more recently Tsiamis used both complex geometric and index theoretic tools in \cite{Tsiamis} to obtain optimal upper bounds for $\sys_2(\omega)$ where $(X,\omega)$ is any compact Kähler manifold with positive scalar curvature as well as refined stable 2-systole bounds for various Riemannian manifolds $(M^n,g)$, including for instance the case where $M^n$ is diffeomorphic to a Fano manifold.

This short note describes an inequality which instead of seeking an upper bound on $\sys_2(\omega)$ for some Kähler manifolds $(X,\omega)$ relates the systoles $\sys_{h_i}(\omega)$ for suitably chosen $h_1,\dots,h_N\in H_2(X)$ in the presence of a product structure and positive scalar curvature.

To state our main result, we need a definition:
\begin{definition}
    A compact complex manifold of complex dimension $n$ $(X,J)$ is said to have \emph{unconditionaly nonpositive total scalar curvature} if for any $\eta$ in $\mathcal{K}(X)$ the Kähler cone of $X$ 
    {(that is, the open convex cone in $H^{1,1}(X)$ of Kähler classes, i.e. classes represented by some Kähler form)},
    \[\int_X c_1(X)\wedge\eta^{n-1}\leq 0.\]
\end{definition}
\begin{remark}
    Our terminology is justified by the fact that for any Kähler metric $\omega$ on $(X,J)$, $\int_X\scal_\omega\omega^n=4\pi n\int c_1(X)\wedge\omega^{n-1}$. In complex dimension $1$, this is equivalent to the non vanishing of the genus of $X$. In higher dimension this is implied by nonpositivity of $c_1(X)$.
\end{remark}

We will focus on compact Kähler manifolds of the form $X=X_0\times\Pi_{k=1}^N\mathbb{CP}^{n_k}$ and denote by $L_k\subset\mathbb{CP}^{n_k}$ a generic complex line and by $\omega_k$ the Fubini-Study metric on $\mathbb{CP}^{n_k}$ with sectional curvature between $\tfrac{1}{4}$ and $1$.
We will show for instance:

\begin{theorem} \label{thm_prod_proj} Let $X_0$ be a compact Kähler manifold of complex dimension $n_0\geq 0$ with unconditionally nonpositive total scalar curvature; assume further that $X_0$ is connected if $n_0=1$. 

    Let $X=X_0\times\Pi_{k=1}^N\mathbb{CP}^{n_k}$.
    Let $\omega$ be any Kähler metric on $X$ such that: \[\int_X\scal_\omega\omega^n\geq S \int_X\omega^n . \]
    where $S=\sum_{k=1}^N {n_k(n_k+1)}{}$.
    
    Then:
    \[\sum_{k=1}^N\frac{n_k(n_k+1)}{\sys_{[L_k]}(\omega)}\geq \sum_{k=1}^N \frac{n_k(n_k+1)}{4\pi} . \]
    
    Moreover if equality is achieved then $\int_X\scal_\omega\omega^n= S \int_X\omega^n$ and $c_1(X_0)$ vanishes.
\end{theorem}
\begin{remark}
    Note that we allow $n_0=0$ in Theorem \ref{thm_prod_proj}. This corresponds to setting $X_0$ to be a point for which the nonpositivity assumption on the curvature is automatic. In this case $X=\Pi_{k=1}^N\mathbb{CP}^{n_k}$.
\end{remark}

Note that the constant $\sum_{k=1}^N {n_k(n_k+1)}{}$ is the scalar curvature of $\Pi_{k=1}^N\mathbb{CP}^{n_k}$ endowed with the product of the Fubini-Study metrics on each factor. 
{Furthermore, due to the simple connectedness of the complex projective spaces and Künneth formula, 
any Kähler class on $X$ decomposes as $\pi_0^*\eta_0+\sum_{k=1}^N a_k\pi_k^*[\omega_k]$ with $\eta_0$ Kähler class on $X_0$, 
$a_k>0$, $\omega_k$ Fubini-Study metric on $\mathbb{CP}^{n_k}$, and $\pi_k$ the projection on $X_0$ if $k=0$, $\mathbb{CP}^{n_k}$ if $k>0$.}

The key ingredient is that the integral of the scalar curvature and the systoles of complex submanifolds can be recovered from purely cohomological data for Kähler manifolds, an idea which is also exploited in \cite{Sha1}, \cite{Sha2} and \cite{Tsiamis}.

This can be easily turned into an optimal upper bound for the $2$-systole of $X$:
\begin{corollary}\label{cor_prod_sys_maj}
    If $(X,\omega)$ satisfies the same assumptions as in Theorem \ref{thm_prod_proj}: 
    \begin{equation}\label{2Lsysbound}
        \min_k \sys_{[L_k]}(\omega) \leq 4\pi.
    \end{equation}
    In particular $\sys_2(\omega)\leq 4\pi$.
    
    Moreover if equality is achieved then $\omega$ is in the Kähler class of $\omega_0+\sum_{k=1}^N\omega_k$ where $\omega_0$ is a Ricci flat Kähler metric on $X_0$.
\end{corollary}

If we moreover assume a pointwise lower bound on the scalar curvature we can improve the rigidity part:
\begin{corollary}\label{cor_2sys_ptwise}
    If $(X,\omega)$ satisfies the same assumptions as in Theorem \ref{thm_prod_proj}, and in addition:
    \[\scal_\omega\geq S\]
    and: 
    \[\min_k \sys_{[L_k]}(\omega) = 4\pi.\]
    
    Then $\omega$ is biholomorphically isometric to $\omega_0+\sum_{k=1}^N\omega_k$ where $\omega_0$ is a Ricci flat Kähler metric on $X_0$.
\end{corollary}

If we assume $X_0$ carries some definite amount of negative curvature, we can get a scale-invariant inequality between various 2-systoles under a nonnegativity assumption on the scalar curvature. To illustrate this phenomenon we prove:

\begin{proposition}\label{nnschypCPn}
    Let $X=\Sigma_g\times\mathbb{CP}^n $ where $\Sigma_g$ is a smooth projective curve of genus $g\geq 2$, and let $\omega$ be a Kähler metric on $X$ such that $\int_X\scal_\omega\omega^{n+1}\geq 0$, then \[\sys_{[L]}(\omega)\leq \frac{n(n+1)}{2(g-1)}\sys_{[\Sigma_g]}(\omega)\] where $L$ is a generic porjective line in $\mathbb{CP}^n$. 
    
    Equality is achieved if and only if $\omega$ is in a positive multiple of the Kähler class of the product metric $\frac{n(n+1)}{2}\omega_{FS}+\omega_h$ where $\omega_{FS}$ is the Fubini-Study metric on $\mathbb{CP}^n$ and $\omega_h$ is the hyperbolic metric on $\Sigma_g$.
    
    If equality is achieved and moreover $\scal_\omega\geq 0$, then $\scal_\omega=0$ and $\omega$ is holomorphically isometric to a positive multiple of $\frac{n(n+1)}{2}\omega_{FS}+\omega_h$.
\end{proposition}
\begin{remark}
    These results were obtained between 2022 and 2024 but were shelved hoping to find some time to improve them in the near future, which didn't happen. 
\end{remark}

\section{Proof of the multi systole inequality}

In this section we prove Theorem \ref{thm_prod_proj}.
Recall that $S=\sum_{k=1}^N n_k(n_k+1)$ is the scalar curvature of $\Pi_{k=1}^N\mathbb{CP}^{n_k}$ endowed with the metric $\omega_1+\dots+\omega_N$ where $\omega_k$ is the Fubini-Study metric on $\mathbb{CP}^{n_k}$ normalised to have sectional curvature between $1/4$ and $1$.

\begin{proof}
Let $\rho$ (resp. $\rho_i$) denote the Ricci form of $\omega$ (resp. $\omega_i$). The Künneth formula gives us that $H_{\mathbb{R}}^{1,1}(X)=\bigoplus_k H^{1,1}_\mathbb{R}(X_k)$. Moreover we can write in $H_{\mathbb{R}}^{1,1}(X)$:
\begin{align*}
    [\rho]&=2\pi c_1(X)\\
    &=2\pi \sum_{k=0}^N c_1(X_k)\\
    &=[\kappa_0]+\frac{n_1+1}{2}[\omega_1]+\dots+\frac{n_N+1}{2}[\omega_N]
\end{align*}
where $\kappa_0$ is a (1,1)-form such that $[\kappa_0]=-2\pi c_1(X_0)$. Note that by our assumption on $X_0$, for any $\omega_0$ in the Kähler cone of $X_0$ we will have:
\[\int_{X_0}\kappa_0\wedge\omega_0^{n_0-1}\leq 0.\]

Let $\eta_0\in H^{1,1}(X_0)$ denote the component of $[\omega]$ in $H^{1,1}(X_0)$. We can then write the cohomology class of the Kähler form $\omega$ in $H_{\mathbb{R}}^{1,1}(X)=\bigoplus_k H^{1,1}_\mathbb{R}(X_k)$ as  $[\omega]=\eta_0+\sum_{k=1}^Na_k[\omega_k]$ for some reals $a_k>0$. Let $n_0=\dim_\mathbb{C}X_0$ and $n=n_0+n_1+\dots+n_N$. We have:
\[\rho\wedge\omega^{n-1}=\frac{1}{2n}\scal_\omega\omega^n
\]

We integrate this identity over $X$ and use that $\int_X\scal_\omega\omega^n\geq S\int_X\omega^n$
 to get :
\[\int_{X}\rho\wedge\omega^{n-1}=\frac{1}{2n}\int_X\scal_\omega\omega^n\geq\frac{1}{2n}\int_X S\omega^n =\frac{S}{2n}\int_{X}\omega^n.\]

On the one hand:
\begin{align*}
    \int_{X}\omega^{n}&=\int_{ X}\left(\eta_0+\sum_{k=1}^Na_k\omega_k\right)^{n} \\
    &=\frac{n!}{n_0!n_1!\cdots n_N!}a_1^{n_1}\cdots a_N^{n_N}\int_{ X}\eta_0^{n_0}\wedge\omega_1^{n_1}\wedge\dots\wedge\omega_N^{n_N}
\end{align*}
by expanding the $n-$th power and keeping only the multiples of the volume form.

On the other hand:
\begin{align*}
    \int_{X}&\rho\wedge\omega^{n-1}=
    \int_{X}\left(\kappa_0+\frac{n_1+1}{2}\omega_1+\dots+\frac{n_N+1}{2}\omega_N\right)\wedge\left(\eta_0+\sum_{k=1}^Na_k\omega_k\right)^{n-1}\\
    \leq &\frac{1}{2}
    \begin{pmatrix}n-1 \\n_0\end{pmatrix}\int_{ X}\left((n_1+1)\omega_1+\dots+(n_N+1)\omega_N\right)\wedge\eta_0^{n_0}\wedge\left(\sum_{k=1}^Na_k\omega_k\right)^{n-n_0-1}\\
\end{align*}
since we left out a positive multiple of:
\[\int_X\kappa_0\wedge \eta_0^{n_0-1}\wedge\omega_1^{n_1}\wedge\dots\wedge\omega_N^{n_N}=\left(\int_{X_0}\kappa_0\wedge \eta_0^{n_0-1}\right)\left(\int_{\Pi_{k=1}^N\mathbb{CP}^{n_k}}\omega_1^{n_1}\wedge\dots\wedge\omega_N^{n_N}
\right)\] 
which is nonpositive by our assumption on $\kappa_0$. 

Expanding $\left(\sum_{k=1}^Na_k\omega_k\right)^{n-n_0-1}$ further we get:
\begin{align*}
    \leq &\frac{1}{2}\begin{pmatrix}n-1 \\n_0\end{pmatrix}\left(\sum_{k=1}^N\frac{(n_k+1)(n-n_0-1)!}{n_1!\dots (n_k-1)!\cdots n_N!}a_1^{n_1}\cdots a_{k}^{n_k-1}\cdots a_N^{n_N}\right)\int_{ X}\eta_0^{n_0}\wedge\dots\wedge\omega_N^{n_N}\\
	\leq &\frac{1}{2}\left(\sum_{k=1}^N\frac{(n_k+1)(n-1)!}{n_0!n_1!\dots (n_k-1)!\cdots n_N!}a_1^{n_1}\cdots a_{k}^{n_k-1}\cdots a_N^{n_N}\right)\int_{ X}\eta_0^{n_0}\wedge\dots\wedge\omega_N^{n_N}
\end{align*}

Hence:
\begin{multline*}
    \frac{1}{2}\left(\sum_{k=1}^N\frac{(n_k+1)(n-1)!}{n_0!n_1!\cdots (n_k-1)!\cdots n_N!}a_1^{n_1}\cdots a_{k}^{n_k-1}\cdots a_N^{n_N}\right)\\ \geq \frac{S}{2n}\frac{n!}{n_0!n_1!\cdots n_N!}a_1^{n_1}\cdots a_N^{n_N}
\end{multline*} 
and:
\[\sum_{k=1}^N\frac{n_k(n_k+1)}{a_k}\geq S\]

Let $L_k\subset {X}$ be a holomorphic $\mathbb{CP}^1$ in the $k$-th $\mathbb{CP}^{n_k}$ factor in $X$. Since $L_k$ is a complex submanifold, then the area of $L_k$ with respect to $\omega$ is \[\mathcal{A}_\omega(L_k)=\int_{L_k}\omega=a_k\int_{L_k}\omega_k=4\pi a_k\] and by Proposition \ref{kaehlercalib} we get $\sys_{[L_{k}]}(\omega)=4\pi a_k$.

This observation together with the expression of $S$ gives the inequality.

Moreover if equality is achieved, we also get that $\int_{X_0}\kappa_0\wedge\eta_0^{n_0-1}=0$. Hence $[\kappa_0]=2\pi c_1(X_0)=0$ by Lemma \ref{lemunsc}.
\end{proof}

\section{2-Systole upper bounds}
We now prove the 2-systole upper bound and the associated rigidity statements.
\begin{proof}[Proof of Corollary \ref{cor_prod_sys_maj}]
    
Writing this as $\sum_{k=1}^N\left(\frac{n_k(n_k+1)}{a_k}-n_k(n_k+1)\right)\geq 0$ we see that one of the $a_k$, say $a_{k_0}=\min_k a_k$, is at most 1.

Since as in the previous proof $4\pi a_{k_0}=\sys_{[L_{k_0}]}(\omega)$ we get the required upper bound.

For the last part of the corollary, assume that $\min_k a_k=1$, then since:
\[\sum_{k=1}^N\left(\frac{n_k(n_k+1)}{a_k}-n_k(n_k+1)\right)\geq 0\] we get that $a_1=\dots=a_N=1$, and thus
\[[\omega]=\eta_0+[\omega_1]+\dots+[\omega_N]=[\omega_0+\omega_1+\dots+\omega_N].\]

\end{proof}

\begin{proof}[Proof of Corollary \ref{cor_2sys_ptwise}.]
    Using the same notations as in the previous proof, let $\hat\omega$ be the product metric $\omega_0+\omega_1+\dots+\omega_N$ which is a constant scalar curvature metric on $X$.
    
    Under the integral bound $\int_X(\scal_\omega-S)\omega^n\geq 0$, we already know that if 
    \[\min_k \sys_{[L_k]}(\omega) = 4\pi\]
    then $\omega$ is in the Kähler class of the product metric $\hat\omega$ and \[\int_X\scal_\omega\omega^n=\int_X S\omega^n=\int_X\scal_{\hat\omega}\omega^n.\]
    
    If we furthermore assume the pointwise bound $\scal_\omega\geq\scal_{\hat\omega}$, we get that $\scal_\omega=\scal_{\hat\omega}$. Hence $\omega$ and $\hat\omega$ are two constant scalar curvature metrics in the same Kähler class. By Berman-Berndtsson's uniqueness of extremal Kähler metrics in a cohomology class \cite{Berman}, $\omega$ and the product metric $\hat\omega$ are isometric through a biholomorphism.
\end{proof}

\section{Product of $\mathbb{CP}^n$ and a projective curve of genus at least 2.}

We now prove Proposition \ref{nnschypCPn}.
\begin{proof}
    Let $\omega_{FS}$ be the Kähler form of the Fubini-Study metric on $\mathbb{CP}^n$ and $\omega_h$ be the Kähler form coming from the hyperbolic metric on $\Sigma_g$.
    Then there exist positive real numbers $a$ and $b$ such that $[\omega]=a[\omega_{FS}]+b[\omega_h]$ in $H^{1,1}(X)$.
    
    The cohomology class of the Ricci form of $\omega$ is given by 
    $[\rho]=\frac{n+1}{2}[\omega_{FS}]-[\omega_h]$.
    
    Then since $\rho\wedge\omega^{n}=\frac{1}{2(n+1)}\scal_\omega\omega^{n+1}$ the hypothesis $\int_X\scal_\omega\omega^{n+1}\geq 0$ gives:
    \[\int_X \left (\frac{n+1}{2}[\omega_{FS}]-[\omega_h]\right)\wedge \left(a[\omega_{FS}]+b[\omega_h]\right)^n\geq 0.\]
    
    We compute :
    \begin{align*}
        \left (\frac{n+1}{2}\omega_{FS}-\omega_h\right)\wedge & \left(a[\omega_{FS}]+b[\omega_h]\right)^n\\
        &=\left(\frac{n+1}{2}\omega_{FS}-\omega_h\right)\wedge(a^n\omega_{FS}^n+na^{n-1}b\omega_{FS}^{n-1}\omega_h)\\
        &=\left(\frac{n(n+1)}{2}a^{n-1}b-a^n\right)\omega_h\wedge\omega_{FS}^n\\
        &=a^{n-1}\left(\frac{n(n+1)}{2}b-a\right)\omega_h\wedge\omega_{FS}^n.
    \end{align*}
    Hence $\int_X\scal_{\omega}\omega^{n+1}\geq 0$ is equivalent to:
    \[a\leq\frac{n(n+1)}{2}b.\]
    Moreover, since $L$ and $\Sigma_g$ are complex curves in $X$, \[\sys_{[L]}(\omega)=\int_L\omega=a\int_L\omega_{FS}=4\pi a\]
    and
    \[\sys_{[\Sigma_g]}(\omega)=\int_{\Sigma_g}\omega=b\int_{\Sigma_g}\omega_{h}=4(g-1)\pi b\] by proposition \ref{kaehlercalib} and Gauss-Bonnet.
    
    Thus $\sys_{[L]}(\omega)\leq\frac{n(n+1)}{2(g-1)}\sys_{[\Sigma_g]}(\omega)$.

    If one has equality, we get that $a=\frac{n(n+1)}{2}b$ which shows that $[\omega]$ is a multiple of $[\frac{n(n+1)}{2}\omega_{FS}+\omega_h]$ in $H^{1,1}(X)$. If the pointwise condition $\scal_\omega\geq 0$ is satisfied, then we get that $\scal_\omega=0$ and the result follows from \cite{Berman} as before.
\end{proof}

\appendix

\section{Some facts from Kähler geometry}

We gather here for convenience of the reader some well known facts from complex geometry. This may also help to dispel some confusions arising from the existence of various normalisations for the Fubini-Study metric and the scalar curvature. Readers seeking further background can read \cite{Huybrechts}.
\subsection{Ricci form and scalar curvature.}
 Let $g_{FS}$ be the Fubini-Study metric on $\mathbb{CP}^n$ with sectional curvature between $1/4$ and $1$, then $\Ric_{g_{FS}}=\frac{n+1}{2}g_{FS}$ and the scalar curvature is thus $n(n+1)$.
\begin{lemma}\label{lemtrace}
    Let $(X,\omega)$ be a Kähler manifold of complex dimension $n$ with Ricci form $\rho$, then $\rho\wedge \omega^{n-1}=\frac{1}{2n}\scal_\omega\omega^n$.
\end{lemma}
\begin{proof}
    We will prove that for any real $(1,1)$-form $\alpha$, 
    $\alpha\wedge \omega^{n-1}= \tfrac{1}{2n}a\omega^n$ where $a$ is the real trace of the symmetric form associated to $\alpha$.
    
    Since $\alpha\mapsto \alpha\wedge\omega^{n-1}$ is linear and $\Lambda^{n,n}X$ is generated by $\omega^n$, there is a linear form $\ell:\Lambda^{1,1}X\to C^{\infty}(X)$ such that $\alpha\wedge\omega^{n-1}=\ell(\alpha)\omega^n$. The symmetries of the problem give that $\ell$ is a multiple of the trace. Evaluating $\ell$ on $\omega$ gives the right factor.

\end{proof}

We will also need the following lemma:
\begin{lemma}\label{lemunsc}
    Let $(X,\omega)$ be a compact Kähler manifold such that:
    \begin{itemize}
        \item For any Kähler metric $\omega\in \mathcal{K}(X)$, $\int_X c_1(X)\wedge\omega^{n-1}\leq 0$,
        \item There exists $\omega_0\in\mathcal{K}(X)$ such that $\int_X c_1(X)\wedge\omega_0^{n-1}= 0$; 
    \end{itemize}
    assume moreover that $X$ is connected if it is a curve. 
    Then $c_1(X)=0$ in $H^{1,1}_\mathbb{R}(X)$.
\end{lemma}
\begin{proof}
    We first treat the $n=1$ case. Note that \[c_1(X)\in H^{1,1}(X)\simeq H^2(X,\mathbb{R})\xrightarrow[\int_X\cdot]{\sim} \mathbb{R},\] 
    hence the equivalence between conditions $\int_X c_1(X)=0$ and $c_1(X)=0$.

    When $n\geq 2$, note that by assumption, the function $\omega\in\mathcal{K}(X)\mapsto\int_X c_1(X)\wedge\omega^{n-1}$ attains its maximum at $\omega_0$.
    
    By differentiation we get that for any $\eta\in H^{1,1}_\mathbb{R}(X)$ we have:
    \begin{equation}\label{diffrel}
        \int_Xc_1(X)\wedge \eta\wedge\omega_0^{n-2}=0.
    \end{equation}
    
   {Now, by Poincaré duality in compact Kähler context, 
    the map 
    \begin{align*}
        H^{n-1,n-1}(X)\times H^{1,1}(X)& \to \mathbb{R}\\ ([\alpha],[\beta])&\mapsto \int_X \alpha\wedge\beta
    \end{align*}  
    is non-degenerate, hence $c_1(X)\wedge\omega_0^{n-2} =0$ in $H^{n-1,n-1}(X)$. 
    By the Hard Lefschetz theorem \cite[\S 0.7]{GrHa}, this in turn implies 
    $c_1(X) = 0$, as claimed. 
    }
\end{proof}
\subsection{Systoles of Kähler manifolds}

The following observation is classic and follows from Wirtinger's inequality and the Kähler condition:
\begin{proposition}\label{kaehlercalib}
    Let $(X,\omega)$ be a Kähler manifold and $Y\subset X$ be a compact complex submanifold. Then for any smooth $\Sigma\in [Y]$, $\mathcal{A}(Y)\leq\mathcal{A}(\Sigma)$.
\end{proposition}
\begin{proof}
    Let $v_\Sigma$ be the area form of $\Sigma$, then Wirtinger inequality implies that for any orthonormal basis $(e_1,\dots,e_{2k})$ of $T_p\Sigma$ :
    \[\omega^k(e_1,\dots,e_{2k})\leq k! v_\Sigma (e_1,\dots,e_{2k})\]
    with equality if and only if $\Sigma$ is a complex submanifold of $X$. This implies that:
    \[\int_\Sigma \frac{\omega^k}{k!}\leq\mathcal{A}(\Sigma).\]
    Since $\Sigma\in [Y]$ and $d(\omega^k)=0$, Stokes theorem implies that:
    \[\int_\Sigma \frac{\omega^k}{k!}=\int_Y \frac{\omega^k}{k!}.\]
    Since $Y$ is a complex submanifold, $\int_Y \frac{\omega^k}{k!}=\mathcal{A}(Y)$ which together with the inequality above gives the result.
\end{proof}

In particular this implies that $\sys_{[Y]}(\omega)=\mathcal{A}(Y)$ for any compact complex submanifold $Y\subset X$.

Moreover, if $\tilde\omega$ is another Kähler metric such that $[\omega]=[\tilde\omega]\in H^{1,1}(X)$ then $\int_Y \frac{\omega^k}{k!}=\int_Y \frac{\tilde\omega^k}{k!}$, this implies that the systoles of the homology classes of complex submanifolds only depend on the cohomology class of $\omega$ in $H^{1,1}(X)$.

~

\noindent
\textbf{Acknowledgments.} We thank the referee for comments which improved the paper. This work was completed during H.A.’s stay at the Intertional Research Laboratory CRM-CNRS (IRL 3457 du CNRS) in Montréal and T.R.’s stay at the Intertional Research Laboratory PIMS-CNRS (IRL 3069 du CNRS) in Vancouver, both on "délégation CNRS" leave from their home universities during academic year 2023-2024; the authors thank these institutions for their support.

\bibliographystyle{alpha}

\bibliography{kaehler}

\end{document}